\documentclass[12,reqno]{amsart}
\usepackage{amsmath, amssymb, amsthm}
\usepackage{url}
\usepackage[breaklinks]{hyperref}

\newcommand{\s}{{\sigma}}

 \renewcommand{\a}{\alpha}

\newcommand{\e}{\epsilon}

\renewcommand{\(}{\left\(}
\renewcommand{\)}{\right\)}
\renewcommand{\[}{\left\[}
\renewcommand{\]}{\right\]}
\renewcommand{\i}{\infty}
\numberwithin{equation}{section}
 \theoremstyle{plain}
\newtheorem{theorem}{Theorem}[section]
\newtheorem{lemma}[theorem]{Lemma}

\newtheorem{corollary}[theorem]{Corollary}
\newtheorem{remark}[theorem]{Remark}

   \makeatletter
\def\proof{\@ifnextchar[{\@oproof}{\@nproof}}
\def\@oproof[#1][#2]{\trivlist\item[\hskip\labelsep\textit{#2 Proof of\
#1.}~]\ignorespaces}
\def\@nproof{\trivlist\item[\hskip\labelsep\textit{Proof.}~]\ignorespaces}

\makeatother

\begin{document}
\title[some $q$-series identities and generalized divisor function ]{On some $q$-series identities related to a generalized divisor function and their implications} 

\author{ Rajat Gupta And Rahul Kumar }
\address{Discipline of Mathematics, Indian Institute of Technology Gandhinagar, Palaj, Gandhinagar 382355, Gujarat, India} 
\email{rajat\_gupta@iitgn.ac.in, rahul.kumr@iitgn.ac.in}
\maketitle

\begin{abstract}
In this article, a $q$-series examined by Kluyver and Uchimura is generalized. This allows us to find generalization of the identities in the random acyclic digraph studied by Simon, Crippa, and Collenberg in 1993. As one of the corollaries of our main theorem, we get results of Dilcher and Andrews, Crippa, and Simon. This main theorem involves a surprising new generalization of the divisor function $\sigma_s(n)$, which we denote by $\sigma_{s,z}(n)$. Analytic properties of $\sigma_{s,z}(n)$ are also studied. As a special case of one of our theorem we obtain result from a recent paper of Bringmann and Jennings-Shaffer.
\end{abstract}

\section{Introduction}
The connection between the divisor function and the coefficients of certain basic hypergeometric series is well-studied. For example, Kluyver obtained in \cite{klu}, for $|q|<1$, namely,  
\begin{align}\label{kluyver}
\sum_{n=1}^{\i}\frac{(-1)^{n-1}q^{n(n+1)/2}}{(1-q^n)(q;q)_n} =\sum_{n=1}^{\i}\frac{q^n}{1-q^n}= \sum_{n=1}^{\i}d(n)q^n, 
\end{align}   
where the notation used above and throughout the paper is as follows, 
\begin{align*}
~~(a;q)_0 &=1;\\
(a;q)_n &:= (1-a)(1-a q)\cdots (1-aq^{n-1}),~ n\geq 1; \\
(a;q)_\i &:= (1-a)(1-aq)\cdots, ~~~~ \mathrm{for} ~|q|<1.
\end{align*} 
Later, Fine in his book \cite[p. 14, Equations (12.4), (12.42)]{fine} and  Zudilin \cite[p.~4]{zudilin} rediscovered \eqref{kluyver}. Uchimura \cite[Theorem 2]{uchi} also gave another equivalent representation for \eqref{kluyver}, that is,
\begin{align}\label{uchimura}
\sum_{n=1}^{\i}nq^n(q^{n+1};q)_\i=\sum_{n=1}^{\i}\frac{(-1)^{n-1}q^{n(n+1)/2}}{(1-q^n)(q;q)_n} =\sum_{n=1}^{\i}\frac{q^n}{1-q^n}.
\end{align}
Identities such as \eqref{uchimura} inherit beautiful combinatorial interpretation and are well-studied in the literature. 

Bressoud and Subbarao \cite{Subba} gave an appealing combinatorial interpretation of the extreme sides of \eqref{uchimura}, namely, 
\begin{align}\label{subba}
\sum_{\pi \in \mathcal{D}_n}(-1)^{\#(\pi) -1}s(\pi) =d(n),
\end{align}
where $\mathcal{D}_n$ is the set of partitions on $n$ into distinct parts, $\#(\pi)$ denotes the number of parts of a partition $\pi$ of $n$, $s(\pi)$ is the smallest part in the a partition $\pi$ of $n$,  and $d(n)$ counts the number of divisors of $n$. The equation \eqref{subba} was also rediscovered by Fokking, Fokking and Wang \cite{FFW}.  Moreover, \eqref{subba} was further generalized by Bressoud and Subbarao \cite{Subba} for 
\begin{align*}
\s_m(n):=\sum_{d|n}d^m;~~ m\in \mathbb{N}\cup\{0\}~and~n \in \mathbb{N}.
\end{align*}
Here, we note that $\s_{0}(n) =d(n).$

A one-variable generalization of $\eqref{kluyver}$ is given in Ramanujan's Notebook \cite[p. 264, Entry 4]{RLIV}, also rediscovered by Uchimura \cite[Equation (3)]{uchi}, namely, for $|zq|<1$ and $z\neq q^{-n},\ n\geq1$, 
\begin{align}\label{zwith}
\sum_{n=1}^{\i}\frac{(-1)^{n-1}z^nq^{n(n+1)/2}}{(1-q^n)(z q;q)_n} =\sum_{n=1}^{\i}\frac{z^n q^n}{1-q^n}. 
\end{align} 
Recently, Andrews, Garvan, and Liang \cite[Theorem 3.5]{AGL} gave a beautiful  generalization of \eqref{uchimura} by generalizing the left-hand side of \eqref{subba} by defining a new weighted--partition sum, 
\begin{align*}
\textup{FFW}(c,n): = \sum_{\pi \in \mathcal{D}_n} (-1)^{\#(\pi) -1}\left(1+c+ \cdots + c^{s(\pi)-1} \right).
\end{align*}
For further discussions and the generalization of \eqref{zwith}, we refer the readers to \cite{AGL}, \cite{dixitmaji}.

K. Dilcher \cite[Equation (4.3), (5.7)]{Dil} obtained an interesting new generalization of \eqref{uchimura}, namely, for $|q|<1$ and $k \in \mathbb{N}$, 
\begin{align}\label{Dilcher1}
\sum_{n=k}^{\i}\binom{n}{k}q^n(q^{n+1};q)_\i = q^{-\binom{k}{2}}\sum_{n=1}^{\i}\frac{(-1)^{n-1}q^{\binom{n+k}{2}}}{(1-q^n)^k(q;q)_n} = \sum_{j_1= 1}^{\i}\frac{q^{j_1}}{1-q^{j_1}}\cdots \sum_{j_k= 1}^{j_{k-1}}\frac{q^{j_k}}{1-q^{j_k}}.
\end{align}
If we let $k=1$ in the above identity, we get \eqref{uchimura} as a special case. He obtained the first equality, and then he proved the equality between the first and the third sum of \eqref{Dilcher1}. To the best of our knowledge, there is no direct proof known of the second equality of \eqref{Dilcher1}.    

By invoking \eqref{Dilcher1} \cite[Section 4]{Dil}, Dilcher gave another generalization of   \eqref{kluyver}, that is,  for $|q|<1$ and $k \in \mathbb{N}$ there exist a polynomial $M_{k}(x_1, x_2,..., ~x_k)$ with rational coefficients such that,
\begin{align}\label{secondgen}
\sum_{n=1}^{\i}\frac{(-1)^{n-1}q^{\binom{n+1}{2}}}{(1-q^n)^k(q;q)_n} = M_{k}\left(S_0(q),~S_1(q),...,~S_{k-1}(q) \right),
\end{align}
where, 
\begin{align}\label{S_k}
S_s(q):=\sum_{n=1}^{\i}\s_s(n)q^n=\sum_{n=1}^{\i}\left(\sum_{d|n}d^s \right)q^n= \sum_{n=1}^{\i}\frac{n^sq^n}{1-q^n}.
\end{align}

Andrews, Crippa and Simon in \cite[Theorem 2.1]{ACS} gave another proof of \eqref{secondgen} and studied its applications in probability theory.

Dixit and Maji, in \cite{dixitmaji}, obtained a more general form of \eqref{kluyver}, namely, for $|a|<1,$ $|b|<1$, $|c|\leq1$ and $|q|<1$,
\begin{align}\label{DM}
\sum_{n=1}^{\i}\frac{(b/a;q)_na^n}{(1-cq^n)(b;q)_n}= \sum_{m=0}^{\i}\frac{(b/c;q)_mc^m}{(b;q)_m}\left(\frac{aq^m}{1-aq^m}- \frac{bq^m}{1-bq^m} \right).
\end{align}
Equation \eqref{DM} also generalizes Ramanujan's identity \cite[p. 354]{Ram}, \cite[p. 263, Entry 3]{RLIV}.

The left-hand side of \eqref{secondgen} is another generalization of \eqref{kluyver} through the variable $k$ . The identities \eqref{kluyver} and \eqref{zwith} are the special cases of \eqref{DM}. We refer the reader to \cite{dixitmaji} for further implications of \eqref{DM}.



In the present paper, we undertake the study of the series
\begin{align}\label{parentsum}
\sum_{n=1}^{\i}\frac{(q/z;q)_nz^n}{(1-q^n)^k(q;q)_n},
\end{align}
whose motivation arose from the aforementioned discussion. Letting $z \to 0$ in \eqref{parentsum}, we get the left-hand side of \eqref{secondgen}, and then if we put $k=1$, we get the left-hand side of \eqref{kluyver}. 

We also note here that a special case to \eqref{parentsum} is studied recently by the first author in \cite[Theorem 1.18]{SOT}. He obtained a sum-of-tails identity, namely, for $|q|<1$, $|c| \leq 1$, $a \in \mathbb{C}$, and $k \in \mathbb{N}$,
\begin{align}
\sum_{n=1}^{\i}\frac{(-a)^{n}q^{\binom{n+1}{2}}}{(1-c q^n)^k(q;q)_n}= \sum_{n=0}^{\i}c^n\binom{k+n-1}{n}\left((aq^n;q)_\i- 1 \right).
\end{align}

One of the goals of this article is to obtain a representation of \eqref{parentsum} generalizing the right-hand side of \eqref{secondgen} and indicate its possible application in acyclic digraphs. Thus our first theorem  is as follows: 
\begin{theorem}\label{maintheorem1}
For $|z|<1,$ $|q|<1$ and $k \in \mathbb{N},$
\begin{align}\label{maintheore1eqn}
\sum_{n=1}^\infty \frac{(q/z;q)_n z^n}{(1-q^n)^k(q;q)_n}=-\frac{(q;q)_\infty}{(z;q)_\infty}\sum_{n=0}^\infty \frac{(z/q;q)_n q^n}{(q;q)_n}{{k+n-1}\choose{k}}.
\end{align}
\end{theorem} 
If we let $z \to 0$ in \eqref{maintheore1eqn}, we get \cite[Equation (9)]{ACS} 
\begin{align*}
\sum_{n=1}^\infty \frac{(-1)^{n-1}q^{n(n+1)/2}}{(1-q^n)^k(q;q)_n}=(q;q)_\infty\sum_{n=0}^\infty \frac{q^n}{(q;q)_n}{{k+n-1}\choose{k}}.
\end{align*}
It is easy to see that the identities \eqref{kluyver} and \eqref{uchimura} are special cases of Theorem \ref{maintheorem1}. 

One non-trivial application of Theorem \ref{maintheorem1} is that it allows us to obtain the generalization of \eqref{secondgen}. In the course of doing so, we stumbled upon an interesting generalization of the divisor function,  which, to the best of our knowledge, does not appear to have been studied. It has  not been concocted artificially; instead, we naturally encountered it while trying to find a generalization of  \eqref{secondgen}. The second goal of this paper is to initiate the study of  this new divisor function.

The topic of $q$-series identities related to divisor functions is of intense research, we refer the reader to a paper of Guo and Zeng \cite{guozeng} for the developments in this area since the appearance of Kluyver's identity \eqref{kluyver}.

Before stating our generalization of \eqref{secondgen}, we define our proposed generalized divisor function by
\begin{align}\label{newdivisor}
\s_{s,z}(n) :=\sum_{d|n}d^s z^d,
\end{align}
where $s,~z \in \mathbb{C}.$

It is straightforward to see that   $\s_{s,z}(n) $ reduces to $\s_{s}(n)$ for $z=1$. 

Very recently, a special case of  \eqref{newdivisor} occurred in the work of Bhoria, Eyyunni, and Maji \cite[Equation 2.5]{maji} in a different context. 

Our first result on the generalized divisor function $\s_{s,z}(n)$ is contained in the following theorem.
\begin{theorem}\label{sumddn}
 Let $s,\ z\in\mathbb{C}$. Then
\begin{align}\label{sumddneqn}
\sigma_{s-1,z}(n)=\frac{1}{n}\sum_{d|n}\varphi(d)\sigma_{s,z}\left(\frac{n}{d}\right),
\end{align}
where $\varphi(d)$ is the Euler totient function \cite[p.~25, Equation (1)]{apostal}.
\end{theorem}

The following well-known result of $\sigma_s(n)$ is a special case of Theorem \ref{sumddn}.
\begin{corollary}\label{spz10}
Let $s\in\mathbb{C}$. Then for $n\geq1$ we have
\begin{align}\label{spz10eqn}
\sigma_s(n)=n^{s-1}\sum_{d|n}\varphi(d)\sigma_{1-s}\left(\frac{n}{d}\right).
\end{align}
\end{corollary} 

Our next result on the generalized divisor function $\s_{s,z}(n)$ is:
\begin{theorem}\label{dirichlet}
Let $s\in\mathbb{C}$ and $|z|\leq1$.  For $\mathrm{Re}(\alpha)>1$, we have 
\begin{align}\label{dirichleteqn}
\sum_{n=1}^\infty\frac{\sigma_{s,z}(n)}{n^{\alpha}}=\zeta(\alpha)\mathrm{Li}_{\alpha-s}(z),
\end{align}
where $\mathrm{Li}_{s}(z)$ is polylogarithm function defined by 
\begin{align}\label{poly}
\mathrm{Li}_{s}(z) =\sum
_{n=1}^{\i}\frac{z^n}{n^s}.
\end{align}
\end{theorem}

The above Theorem gives the following well-known result as its special case.
\begin{corollary}\label{zetadircor}
For $\textup{Re}(\alpha)>\mathrm{max}\{1,1+\mathrm{Re}(s)\}$, we have 
\begin{align}\label{zetadir}
\sum_{n=1}^\infty\frac{\sigma_{s}(n)}{n^{\alpha}}=\zeta(\alpha)\zeta(\alpha-s).
\end{align}
\end{corollary}

The average order of any arithmetical function is always desirable. The average order of $\s_{s,z}(n)$ obtained in the next theorem.
\begin{theorem}\label{s=0caseofgdf}
Let $0<z\leq1,\ s<0$. Then for $x\geq1$, we have
\begin{align}
\sum_{n\leq x}\sigma_{s,z}(n)&=-x^{1+s}E_{1-s}(-x\log(z))+\frac{1}{2}x^{1+s}E_{-s}(-x\log(z))+x\mathrm{Li}_{1-s}(z)-\frac{1}{2}\mathrm{Li}_{-s}(z)+O(x^\beta),\nonumber
\end{align}
where $\beta=\mathrm{max}\left\{0,x^{1+s}E_{-s}(-x\log(z))\right\}$ and $E_{\nu}(z)$ is Exponential integral which is defined by \cite[p.~228, Equation (5.1.4)]{handbook}
\begin{align}
E_{\nu}(z):=\int_1^\infty \frac{e^{-zt}}{t^{\nu}}\ dt, \ \mathrm{Re}(z)>0.
\end{align}
\end{theorem}

Many results similar to Theorem \ref{s=0caseofgdf} and their special cases are obtained in Section \ref{pgdf}. 

The first appearance of the generalized divisor function $\s_{s,z}(n)$ occurs in Theorem \ref{maintheorem2}. Before stating this theorem, we need to define $\mathfrak{S}_{s,z}(q)$ by

\begin{align}\label{scripts}
\mathfrak{S}_{s,z}(q):=S_s(q)- S_{s,z}(q),
\end{align}
where, 
\begin{align}\label{Ssz}
S_{s,z}(q) := \mathrm{Li}_{-s}(z)+\sum_{n=1}^{\i}\s_{s,z}(n)q^n =\mathrm{Li}_{-s}(z)+\sum_{n=1}^{\i}\frac{n^s z^n q^n}{1-q^n}.
\end{align}
If we let $z =0$ in \eqref{scripts} then $\mathfrak{S}_{s,0}(q)=S_s(q)$, where $S_s(q)$ is  defined in \eqref{S_k}.

Now, we are all set to state our next theorem which involves the generalized divisor function $\sigma_{s,z}(n)$. 
\begin{theorem}\label{maintheorem2}
Let $\mathfrak{S}_{s,z}(q)$ be defined in \eqref{scripts}. Then for $|q|<1, |z|<1$ and $k \in \mathbb{N}$, there exist a polynomial $M_{k}(x_1,x_2,...,x_k)$ with rational coefficients, such that 
\begin{align}\label{maintheorem2eqn}
\sum_{n=1}^{\infty}\frac{(q/z;q)_n z^n}{(1-q^n)^k(q;q)_n} = -M_{k}\left(\mathfrak{S}_{0,z}(q),  \mathfrak{S}_{1,z}(q), \mathfrak{S}_{2,z}(q) ,...,\mathfrak{S}_{k-1,z}(q) \right).
\end{align}
\end{theorem}

It is easy to see that as $z\to0$ in the above theorem we get \eqref{secondgen}. 

As a special case of Theorem \ref{maintheorem2} for $k=2,$ we derive the following interesting result.
\begin{corollary}\label{hjmeqn}
For $|z|<1$, and $|q|<1,$
\begin{align}\label{hjm}
2\sum_{n=1}^{\infty}\frac{(q/z;q)_n z^n}{(1-q^n)^2(q;q)_n} =\frac{z(2-z)}{(1-z)^2}-\sum_{n=1}^{\i}\frac{(1-z^n)}{(1-q^n)}(n+1)q^n - \left(\frac{z}{(1-z)}-\sum_{n=1}^{\i}\frac{(1-z^n)}{(1-q^n)}q^n  \right)^2.
\end{align}
\end{corollary}

\begin{remark}
If we differentiate both sides of \eqref{hjm} with respect to $z$ and then take $z\to q$ in the resulting expression then we deduce the following elegant $q$-series identity
\begin{align}
2\sum_{n=1}^\infty \frac{q^n}{(1-q^n)^3}=\frac{2q}{(1-q^n)^3}+\sum_{n=1}^{\infty}\frac{n(n+1)q^{2n}}{(1-q^n)}.
\end{align}
\end{remark}

Simon, Crippa, and Collenberg \cite{scc} showed that the expectation and variance of a certain random variable arising from acyclic digraphs can also be represented in terms of divisor function. One of their results is as follows: For fixed $n$, if the random variable $\gamma_n^*$ is defined by the number of vertices reachable from the vertex 1 then
\begin{align}\label{variance}
\lim_{n\to\infty}(n-E(\gamma_n^*))=\sum_{j=1}^{\infty}\sum_{d|n}q^j,
\end{align}
and 
\begin{align*}
\lim_{n\to\infty}\mathrm{Var}(\gamma_n^*)=\sum_{j=1}^{\infty}\sum_{d|j}dq^j.
\end{align*}

Later in \cite[Theorem 3.1]{ACS}, authors proved the following theorem by invoking their result \eqref{secondgen}.
\begin{theorem}
Let $a_n(q)$ be a polynomial in q defined by the recursive equation
$$a_{n}(q):= f(n)+\left(1-q^{n-1}\right)a_{n-1}(q), \qquad n \geq 1,$$
with $a_{0}(q)=0.$ Then there exist rational coefficients $h_{j}$ such that 
\begin{align}
\lim_{n\to \infty}\left\{\sum_{j=1}^{n}f(j)-a_n(q)\right\}=\sum_{j=1}^{\infty}h_jM_{j};
\end{align}
where 
\begin{align*}
\mathrm{for}~j\geq 2,~ h_j =\sum_{i \geq j-1}(-1)^{i-j+1}\binom{i-1}{j-2}i!\sum_{k \geq i}c_k \tilde{s}(k,i);~~ h_1 =c_0,
\end{align*}
and $\tilde{s}(k,i)$ are Stirling numbers of the second kind.
\end{theorem}

It is clear that if we take $f(n)=1$ in the above theorem we get \eqref{variance}. For further details on this, we refer the reader to \cite{ACS}.

In their paper \cite[p.~56]{ACS}, authors posed a question of obtaining a similar result for $f(n)$ being periodic function. This question is affirmatively answered by Bringmann and Jennings-Shaffer in \cite[Theorem 1.3]{bring}. In the same paper, they have also provided a similar result for $f(n) = b^n,~ b \in\mathbb{C}\backslash\{1\}$. Our Theorem \ref{maintheorem2} also helps us in finding the following elegant  generalization of \cite[Theorem 1.3]{bring}.
\begin{theorem}\label{bringperiod}
Let $f(n)$ be a periodic sequence with period $N$ and $a_n(z,q)$ is the sequence such that
\begin{align*}
a_{n}(z,q):= \left(1-z/q\right)f(n)+\left\{1-\left(1-z/q\right)q^{n-1}\right\}a_{n-1}(z,q),\qquad a_0(z,q)=0.
\end{align*}
If 
$
\displaystyle c_k := \frac{1}{N}\sum_{1\leq j \leq N}f(j)\zeta^{(1-j)k}_{N},$ then for $|z|<1$ and $|q|<1,$ we have
\begin{align}\label{bringperiodeqn}
    \lim_{n\to \infty}&\left(\left(1-z/q\right)\sum_{j=1}^{n}f(j)-a_{n}(z,q)\right) \nonumber \\
    &=\left(1-z/q\right)\Bigg(c_{0}\mathfrak{S}_{0,z}(q)+\sum_{1\leq k\leq N-1}\frac{c_k}{1-\zeta^{k}_N}-\frac{(q;q)_\i}{(z;q)_\i}\sum_{1\leq k\leq N-1}c_k\frac{(z \zeta^k_N;q)_\i}{(\zeta^k_N;q)_\i}\Bigg).
\end{align}
\end{theorem}

As an application of Theorem \ref{bringperiod} and Theorem \ref{gen1.4} we obtain the following generalization of \cite[Corollary 4.1]{bring}\footnote{In the first equality of Corollary 4.1 of \cite{bring}, $-\frac{1}{2}-\frac{(q;q)_{\infty}}{2(-q;q)_{\infty}}$ must be $-\frac{1}{2}+\frac{(q;q)_{\infty}}{2(-q;q)_{\infty}}$.}.
\begin{corollary}\label{gk10}
Let $a_n(z,q)$ be the sequence defined as $a_{n}(z,q):= \left(1-z/q\right)(-1)^n+\left\{1-\left(1-z/q\right)q^{n-1}\right\}\newline a_{n-1}(z,q),~ a_0(z,q)=0$. For $|z|<1$ and $|q|<1$, we have
\begin{align}
\lim_{n\to \infty}\left((1-z/q)\sum_{j=1}^{n}(-1)^j-a_{n}(z,q)\right)=-\frac{(q,q)_{\infty}}{(z;q)_{\infty}}\sum_{n=0}^\infty\frac{(z/q;q)_{2n+1}q^{2n+1}}{(q;q)_{2n+1}}=-\frac{1}{2}+\frac{1}{2}\frac{(q;q)_{\infty}(-z;q)_{\infty}}{(-q;q)_{\infty}(z;q)_{\infty}}.
\end{align}
\end{corollary}

This paper is organised as follows. We first collect some known results from the literature in Section \ref{lit} which will be employed in the sequel. Section \ref{prrofs} is devoted to proving Theorems \ref{maintheorem1}, Theorem \ref{maintheorem2} and to obtaining several lemmas derived to prove Theorem \ref{maintheorem2}. In section \ref{extn}, Theorem \ref{bringperiod}, its Corollary \ref{gk10}, and several other results are proved. Section \ref{pgdf} contains the theory of the generalized divisor function $\sigma_{s,z}(n)$, namely, Theorem \ref{sumddn}, Theorem \ref{dirichlet}, Theorem \ref{s=0caseofgdf}. Several other properties of $\sigma_{s,z}(n)$ are also  obtained in this section. We conclude the paper with proposing some questions in Section \ref{conclude}.

\section{Preliminaries}\label{lit}
The $q$-Gauss sum identity \cite[p.~354, Equation (II.8)]{gasper} is given by  
\begin{align}\label{qgauss}
{}_2\phi_1(a,b;c;q,c/ab)=\frac{(c/a;q)_\infty(c/b;q)_\infty}{(c;q)_\infty(c/(ab);q)_\infty}.
\end{align}
We note down the $q$-binomial theorem \cite[p.~8, Equation (1.3.2)]{gasper}, for $|z|<1,\ |q|<1$ and $a\in\mathbb{C}$:
\begin{align}\label{qanalog}
\sum_{n=0}^\infty\frac{(a;q)_n}{(q;q)_n}z^n=\frac{(az;q)_\infty}{(z;q)_\infty}.
\end{align}
An equivalent version of the $q$-binomial theorem \eqref{qanalog} is \cite[p.~9, Equation (1.3.8)]{berndt} 
\begin{align}\label{analogqbinom}
\sum_{n=0}^\infty\frac{(a/b;q)_{n}}{(q;q)_{n}}(by)^n=\frac{(ay;q)_{\infty}}{(by;q)_{\infty}}.
\end{align}
We also record the Chu-Vandermonde identity \cite{gkp}
\begin{align}\label{chu}
\sum_{r=1}^k{{n}\choose{r}}{{k-1}\choose{k-r}}={{k+n-1}\choose{k}}.
\end{align}

\section{Proofs of $q$-series identities}\label{prrofs}
We begin this section with a proof of Theorem \ref{maintheorem1}.
\begin{proof}[Theorem \textup{\ref{maintheorem1}}][]
Consider
\begin{align}\label{r1}
R(z,\xi):&=\sum_{k=1}^\infty\sum_{n=1}^\infty \frac{(q/z;q)_n z^n}{(q;q)_n(1-q^n)^k}\xi^k.
\end{align}
Then by using binomial theorem in \eqref{r1}, we see that
\begin{align}\label{r56}
R(z,\xi)&=\frac{\xi}{1-\xi}\sum_{n=1}^\infty \frac{(q/z;q)_n z^n}{(q;q)_n\left(1-q^n/(1-\xi)\right)}\nonumber\\
&=-\sum_{n=1}^\infty \frac{\left(1/(1-\xi);q\right)_n(q/z;q)_n z^n}{(q;q)_n\left(q/(1-\xi);q\right)_n}\nonumber\\
&=1-{}_2\phi_1\left(\frac{1}{1-\xi},\frac{q}{z};\frac{q}{1-\xi};q,z\right),
\end{align}
Upon using \eqref{qgauss} in \eqref{r56}, we get
\begin{align}\label{r2}
R(z,\xi)&=1-\frac{(q;q)_\infty\left(z/(1-\xi);q\right)_\infty}{(z;q)_\infty\left(q/(1-\xi);q\right)_\infty}.
\end{align}
Equation \eqref{qanalog} and \eqref{r2} implies that
\begin{align}
R(z,\xi)&=1-\frac{(q;q)_\infty}{(z;q)_\infty}\sum_{n=0}^\infty\frac{(z/q;q)_n}{(q;q)_n}\left(\frac{q}{1-\xi}\right)^n\nonumber\\
&=1-\frac{(q;q)_\infty}{(z;q)_\infty}\sum_{n=0}^\infty\frac{(z/q;q)_n}{(q;q)_n}q^n\sum_{k=0}^\infty{{k+n-1}\choose{k}}\xi^k.
\end{align}
Now use the definition \eqref{r1} of $R(z,\xi)$ in the above equation and then compare the coefficients of $\xi^k,\ k\geq1$ on both sides of the above equation to arrive at \eqref{maintheore1eqn}.
\end{proof}

To prove Theorem \ref{maintheorem2}, we need several lemmas which we prove in the sequel below.

\begin{lemma}\label{der}
For $r \in \mathbb{N}$, $|z|<1$ and $|q|<1$
\begin{align*}
\frac{d^r}{d{\e}^r}\left(\frac{(\e z;q)_\infty}{(\e q;q)_\infty}\right)\Bigg|_{\e \to 1}=r!\sum_{n=0}^{\infty} \binom{n}{r} \left(z/q;q\right)_n \frac{q^n}{(q;q)_n}.
\end{align*}
\end{lemma}

\begin{proof}
An application of \eqref{qanalog} implies that
\begin{align*}
\frac{d^r}{d{\e}^r}\frac{(\e z;q)_\infty}{(\e q;q)_\infty}&= \frac{d^r}{d{\e}^r}\sum_{n=0}^{\infty}\frac{(z/q;q)_n}{(q;q)_n}\e^n q^n \\
&=\sum_{n=0}^{\infty}\frac{(z/q;q)_n}{(q;q)_n}n(n-1)....(n-r+1)\e^{n-r} q^n\\
&=r!\sum_{n=0}^{\infty}\frac{(z/q;q)_n}{(q;q)_n}\binom{n}{r}\e^{n-r} q^n.
\end{align*}
Letting $\e \to 1$ on both sides of the above equation, we arrive at the statement of our lemma.  
\end{proof}

Define a new function, 
\begin{align}\label{T}
T_{r,z}=T_{r,z}(\e, q):
&=\sum_{n=1}^{\infty}\frac{q^{nr}}{(1-\e q^n)^r}-\sum_{n=0}^{\infty}\frac{z^rq^{nr}}{(1-\e z q^n)^r}.
\end{align}
A simple observation leads to 
\begin{align}
\frac{d}{d\e}T_{r,z}(\e,q)=r T_{r+1,z}(\e,q). \label{derivative}
\end{align}

\begin{lemma}\label{der2}
For each $k\in \mathbb{N}, z \in \mathbb{C}$ and $|q|<1$ there exists a  $k$-degree rational polynomial $N_k(x_1, x_2,..,x_k)$ such that 
\begin{align}
\frac{d^k}{d\e^k}\frac{(\e z;q)_\infty}{(\e q;q)_\infty}=\frac{(\e z;q)_\infty}{(\e q;q)_\infty}N_{k}\left(T_{1,z}, T_{2,z},..,T_{k,z}\right).
\end{align}
\end{lemma}

\begin{proof}
Note that
{\allowdisplaybreaks\begin{align}\label{x5}
\frac{d}{d\e}\frac{(\e z;q)_\infty}{(\e q;q)_\infty}&=\frac{d}{d\e}\left((1-\e z)\prod_{j=1}^\infty\frac{1-\e zq^j}{1-\e q^j}\right)\nonumber\\
&=(1-\e z)\frac{d}{d\e}\left(\exp\left(\log\left(\prod_{j=1}^\infty\frac{1-\e zq^j}{1-\e q^j}\right)\right)\right)-\frac{z}{1-\e z}\frac{(\e z;q)_\infty}{(\e q;q)_\infty}\nonumber\\
&=(1-\e z)\prod_{j=1}^\infty\frac{1-\e zq^j}{1-\e q^j}\frac{d}{d\e}\left(\sum_{j=1}^\infty\log\left(\frac{1-\e zq^j}{1-\e q^j}\right)\right)-\frac{z}{1-\e z}\frac{(\e z;q)_\infty}{(\e q;q)_\infty}\nonumber\\
&=\frac{(\e z;q)_\infty}{(\e q;q)_\infty}\sum_{j=1}^\infty\left\{\frac{-zq^j}{1-\e zq^j}-\frac{-q^j}{1-\e q^j}\right\}-\frac{z}{1-\e z}\frac{(\e z;q)_\infty}{(\e q;q)_\infty}\nonumber\\
&=\frac{(\e z;q)_\infty}{(\e q;q)_\infty}\left(\sum_{j=1}^\infty\frac{q^j}{(1-\e q^j)}-\sum_{j=0}^\infty\frac{zq^j}{(1-\e zq^j)}\right)\nonumber\\
&=\frac{(\e z;q)_\infty}{(\e q;q)_\infty}T_{1,z}(\e, q),
\end{align}}
where $T_{1,z}(\e, q)$ is defined in \eqref{T}. If we take $N_{1}(x_1):=x_1$ then the above equation leads to
\begin{align}\label{x51}
\frac{d}{d\e}\frac{(\e z;q)_\infty}{(\e q;q)_\infty}=\frac{(\e z;q)_\infty}{(\e q;q)_\infty}N_{1}\left(T_{1,z}\right).
\end{align}
Again differentiating \eqref{x51} with respect to $\e$, we see that
\begin{align*}
\frac{d^2}{d\e^2}\frac{(\e z;q)_\infty}{(\e q;q)_\infty}&=\frac{(\e z;q)_\infty}{(\e q;q)_\infty}T^2_{1,z}(\e, q)+\frac{(\e z;q)_\infty}{(\e q;q)_\infty}\frac{d}{d\e}T_{1,z}(\e, q).
\end{align*}
Invoking \eqref{derivative} in the above equation, we have
\begin{align*}
\frac{d^2}{d\e^2}\frac{(\e z;q)_\infty}{(\e q;q)_\infty}&=\frac{(\e z;q)_\infty}{(\e q;q)_\infty}T^2_{1,z}(\e, q)+\frac{(\e z;q)_\infty}{(\e q;q)_\infty}T_{2,z}(\e, q)=\frac{(\e z;q)_\infty}{(\e q;q)_\infty}N_{2}(T_{1,z}, T_{2,z}),
\end{align*}
where $N_{2}(x_1,x_2):= x^2_1+x_2.$ Thus by using induction and \eqref{derivative}, we conclude the proof.
\end{proof}

The next lemma gives a representation for $T_{r,z}(1, q)$.
\begin{lemma}\label{t1q}
Let $\mathfrak{S}_{s,z}(q)$ be defined in \eqref{scripts}. For $r \in \mathbb{N}$ there exists a rational constant $c_{r,h}$ for $0 \leq h \leq r-1$ such that
\begin{align}\label{t1qeqn}
T_{r,z}(1, q)=\sum_{h=0}^{r-1}c_{r,h}\mathfrak{S}_{h,z}(q).
\end{align}
\end{lemma}

\begin{proof}
From \eqref{T}, we have
\begin{align}\label{t1q0}
T_{r,z}(1, q)&=\sum_{n=1}^{\infty}\frac{q^{nr}}{(1-q^n)^r}-\sum_{n=0}^{\infty}\frac{z^rq^{nr}}{(1-z q^n)^r}.
\end{align}
We invoke \cite[Lemma 2.5]{ACS} in \eqref{t1q0} to see that
\begin{align}\label{t1q00}
T_{r,z}(1, q)&=T_{r}(1, q)-\sum_{n=0}^{\infty}\frac{z^rq^{nr}}{(1-z q^n)^r}\nonumber\\
&=\sum_{j=0}^{r-1}c_{r,j}S_j(q)-\sum_{n=0}^{\infty}\frac{z^rq^{nr}}{(1-z q^n)^r}.
\end{align}
Upon using the binomial theorem, we see that
\begin{align}\label{t1q1}
\sum_{n=0}^\infty \frac{z^rq^{nr}}{(1-zq^n)^r}&=\sum_{n=0}^\infty\frac{zq^n(1-(1-zq^{n})^{r-1}}{(1-zq^n)^r}\nonumber\\
&=\sum_{n=0}^\infty\frac{zq^n}{(1-zq^n)^r}\sum_{j=0}^{r-1}{{r-1}\choose j}(-1)^j(1-zq^n)^j\nonumber\\
&=z\sum_{j=0}^{r-1}{{r-1}\choose j}(-1)^j\sum_{n=0}^\infty\frac{q^n}{(1-zq^n)^{r-j}}.
\end{align}
We again employ the binomial theorem in \eqref{t1q1} to get
\begin{align}\label{t1q2}
\sum_{n=0}^\infty \frac{z^rq^{nr}}{(1-zq^n)^r}&=z\sum_{j=0}^{r-1}{{r-1}\choose j}(-1)^j\sum_{n=0}^\infty\sum_{m=0}^\infty{{r-j+m-1}\choose{r-j-1}}q^{n(1+m)}z^m\nonumber\\
&=z\sum_{j=0}^{r-1}{{r-1}\choose j}(-1)^j\sum_{n=0}^\infty\sum_{m=0}^\infty\frac{q^{n(1+m)}z^m}{(r-j-1)!}(m+1)(m+2)...(m+r-j-1)\nonumber\\
&=z\sum_{j=0}^{r-1}\frac{(-1)^j}{(r-j-1)!} {{r-1}\choose j}\sum_{n=0}^\infty\sum_{m=1}^\infty q^{nm}z^{m-1}(m)(m+1)...(m+r-j-2)\nonumber\\
&=\sum_{j=0}^{r-1}\frac{(-1)^{r-1}}{(r-j-1)!} {{r-1}\choose j}\sum_{n=0}^\infty\sum_{m=1}^\infty q^{nm}z^{m}(-m)(-m-1)...(-m-r+j+2).
\end{align}
Now note that generating function for Stirling numbers is given by
\begin{align}\label{str}
x(x-1)\dotsm(x-j+1)=\sum_{i=0}^js(j,i)x^i.
\end{align}
From \eqref{t1q2} and \eqref{str}, we have
\begin{align}\label{t1q000}
\sum_{n=0}^\infty \frac{z^rq^{nr}}{(1-zq^n)^r}&=\sum_{j=0}^{r-1}\frac{(-1)^{r-1}}{(r-j-1)!} {{r-1}\choose j}\sum_{n=0}^\infty\sum_{m=1}^\infty q^{nm}z^{m}\sum_{i=0}^{r-j-1}(-1)^im^is(r-j-1,i)\nonumber\\
&=\sum_{j=0}^{r-1}\frac{(-1)^{r-1}}{(r-j-1)!} {{r-1}\choose j}\sum_{i=0}^{r-j-1}(-1)^is(r-j-1,i)\sum_{n=0}^\infty\sum_{m=1}^\infty m^iz^{m}q^{nm}\nonumber\\
&=\sum_{j=0}^{r-1}\frac{(-1)^{r-1}}{(r-j-1)!} {{r-1}\choose j}\sum_{i=0}^{r-j-1}(-1)^is(r-j-1,i) S_{i,z}(q)\nonumber\\
&=\sum_{i=0}^{r-1}S_{i,z}(q)\sum_{j=0}^{r-i-1}\frac{(-1)^{i+r-1}}{(r-j-1)!} {{r-1}\choose j}s(r-j-1,i)\nonumber\\
&=\sum_{i=0}^{r-1}c_{r,i}S_{i,z}(q),
\end{align}
where $S_{i,z}(q)$ is defined in \eqref{Ssz}. Now combine \eqref{t1q00} and \eqref{t1q000} to arrive at \eqref{t1qeqn}.
\end{proof}

We are now ready to prove Theorem \ref{maintheorem2}. 
\begin{proof}[Theorem \textup{\ref{maintheorem2}}][]
Invoke Theorem \ref{maintheorem1} and \eqref{chu} so that 
\begin{align}\label{x2}
\sum_{n=1}^\infty\frac{(q/z;q)_nz^n}{(q;q)_n(1-q^n)^k}&=-\frac{(q;q)_\infty}{(z;q)_\infty}\sum_{n=0}^\infty\frac{(z/q;q)_nq^n}{(q;q)_n}\sum_{r=1}^k{n\choose r}{{k-1}\choose{k-r}}\nonumber\\
&=-\frac{(q;q)_\infty}{(z;q)_\infty}\sum_{r=1}^k{{k-1}\choose{k-r}}\sum_{n=0}^\infty\frac{(z/q;q)_nq^n}{(q;q)_n}{n\choose r}.
\end{align}
Then employ Lemma \ref{der} in \eqref{x2} in the first step and Lemma \ref{der2} in the second step below so as to have
\begin{align}\label{x3}
\sum_{n=1}^\infty\frac{(q/z;q)_nz^n}{(q;q)_n(1-q^n)^k}&=-\frac{(q;q)_\infty}{(z;q)_\infty}\sum_{r=1}^k{{k-1}\choose{k-r}}\frac{1}{r!}\left[\frac{d^r}{d\epsilon^r}\frac{(qz;q)_\infty}{(\epsilon q;q)_\infty}\right]_{\epsilon=1}\nonumber\\
&=-\frac{(q;q)_\infty}{(z;q)_\infty}\sum_{r=1}^k{{k-1}\choose{k-r}}\frac{1}{r!}\left[\frac{(\epsilon z;q)_\infty}{(\epsilon q;q)_\infty}N_r[T_{1,z}(\epsilon,q),...,T_{r,z}(\epsilon,q)]\right]_{\epsilon=1}\nonumber\\
&=-\sum_{r=1}^k{{k-1}\choose{k-r}}\frac{1}{r!}N_r[T_{1,z}(1,q),...,T_{r,z}(1,q)].
\end{align}
 Now an application of Lemma \ref{t1q} in \eqref{x3} implies that
\begin{align}
\sum_{n=1}^\infty\frac{(q/z;q)_nz^n}{(q;q)_n(1-q^n)^k}&=-\sum_{r=1}^k{{k-1}\choose{k-r}}\frac{1}{r!}N_r\Big[c_{1,0}\mathfrak{S}_{0,z}(q),c_{2,0}\mathfrak{S}_{0,z}(q)+c_{2,1}\mathfrak{S}_{1,z}(q),...,c_{r,0}\mathfrak{S}_{0,z}(q)\nonumber\\
&\qquad\qquad\qquad +c_{r,1}\mathfrak{S}_{1,z}(q)+...+c_{r,r-1}\mathfrak{S}_{r-1,z}(q)\Big]\nonumber\\
&=-M_k\left(\mathfrak{S}_{0,z}(q),\mathfrak{S}_{1,z}(q),...,\mathfrak{S}_{k-1,z}(q)\right).\nonumber
\end{align}
This concludes the proof.
\end{proof}

We give explicit form of $M_1(x_1)$ as well as  $M_2(x_1, x_2)$.

Letting $k=1$ in \eqref{x3} and then using the fact $N_{1}(x) =x$, obtained in Lemma \ref{der2}, to arrive, 
\begin{align*}
\sum_{n=1}^\infty\frac{(q/z;q)_nz^n}{(q;q)_n(1-q^n)} = -N_{1}[T_{1,z}(1,q)] =-T_{1,z}(1,q).
\end{align*} 
Now, employing Lemma \ref{t1q} in the above equation with the fact that $c_{1,0} =1$, we obtain, 
\begin{align*}
\sum_{n=1}^\infty\frac{(q/z;q)_nz^n}{(q;q)_n(1-q^n)} = -\mathfrak{S}_{0,z}(q).
\end{align*}
Hence $M_{1}(x_1) = x_1,$ and with the help of $N_{1}(x_1)=x_1$, $N_{2}(x_1,~x_2)= x_1^2 +x_2$, we have $M_2(x_1, ~x_2):=\frac{1}{2}(x_1^2+x_1+x_2).$ Similarly $M_{k}(x_1, x_2,\cdots, x_k)$ for $k \geq 3$ can be defined.

\begin{proof}[Corollary \textup{\ref{hjmeqn}}][]
Let $k=2$ in Theorem \ref{maintheorem2} and use the fact that $M_2(x_1,x_2)=\frac{1}{2}(x_1^2+x_1+x_2)$ with letting $x_1=\mathfrak{S}_{0,z}(q)$ and $x_2=\mathfrak{S}_{1,z}(q)$. Then simplify to arrive at \eqref{hjm}.
\end{proof}

\section{Extensions of $q$-series identities arising from random graphs}\label{extn}

Let us consider the general sequence
\begin{align}\label{anzq}
a_{n}(z,q):= \left(1-z/q\right)f(n)+\left\{1-\left(1-z/q\right)q^{n-1}\right\}a_{n-1}(z,q), ~~ n \geq 1 ~and~ a_0(z,q)=0,
\end{align}
which satisfies the properties given in the lemma below. 
\begin{lemma}\label{anqrepr}
For $n \in \mathbb{N}$,
\begin{enumerate}
\item $\begin{aligned}[t]
   a_n(z,q) = \left(1-z/q\right)\sum_{i=1}^{n}f(n+1-i)\prod_{j=1}^{i-1}\left\{1-\left(1-z/q\right)q^{n-j}\right\},
\end{aligned}$
\item $\begin{aligned}[t]
    a_n(z,q) = \left(1-z/q\right)\left(\sum_{i=1}^{n}f(i)-\sum_{j=1}^{n-1}a_j(z,q) q^j \right).\label{angrepreqn2}
\end{aligned}$
\end{enumerate}
\end{lemma}

\begin{proof}
To prove $(1)$, we use induction on $a_n(z,q)$. Let $n=1$ in \eqref{anzq} so that
$$a_1(z,q)=\left(1-z/q\right)f(1),$$
which shows that (1) is true for $n=1$. 

Suppose $(1)$ is valid for $n=k$. We show that it holds for $n=k+1$. Observe that
\begin{align*}
a_{k+1}(z,q)&=\left(1-z/q\right)f(k+1)+\left\{1-\left(1-z/q\right)q^{k}\right\}a_{k}(z,q)\\
&=\left(1-z/q\right)f(k+1)+\left\{1-\left(1-z/q\right)q^{k}\right\}\left[ \left(1-z/q\right)\sum_{i=1}^{k}f(k+1-i)\prod_{j=1}^{i-1}\left\{1-\left(1-z/q\right)q^{k-j}\right\} \right] \\
&=\left(1-z/q\right)f(k+1)+\left[ \left(1-z/q\right)\sum_{i=1}^{k}f(k+1-i)\prod_{j=1}^{i}\left\{1-\left(1-z/q\right)q^{k-j+1}\right\} \right]\\
&= \left(1-z/q\right)\sum_{i=0}^{k}f(k+1-i)\prod_{j=1}^{i}\left\{1-\left(1-z/q\right)q^{k-j+1}\right\} \\
&=\left(1-z/q\right)\sum_{i=1}^{k+1}f(k+2-i)\prod_{j=1}^{i-1}\left\{1-\left(1-z/q\right)q^{k-j+1}\right\}.
\end{align*}
This shows (1) holds for $n=k+1$ too. T`herefore, by the principal of mathematical induction, (1) follows for all $n\in\mathbb{N}$.

We again use the induction to prove $(2)$. By using the definition \eqref{anzq} of $a_n(z,q)$, it is easy to see that (2) holds for $n=1$.  Let's assume that $(2)$ is valid for $n=k$. Then for $n=k+1$, we have
\begin{align*}
a_{k+1}(z,q)&=\left(1-z/q\right)f(k+1)+\left\{1-\left(1-z/q\right)q^{k}\right\}a_{k}(z,q)\\
&=\left(1-z/q\right)f(k+1)+a_{k}(z,q)-\left(1-z/q\right)q^{k}a_{k}(z,q).
\end{align*}
Now use the assumption that $a_k(z,q)$ satisfies $(2)$ in the middle term of the above equation. Then 
\begin{align*}
a_{k+1}(z,q)&=\left(1-z/q\right)f(k+1)+\left(1-z/q\right)\left(\sum_{i=1}^{k}f(i)-\sum_{j=1}^{k-1}a_j(z,q) q^j \right)-\left(1-z/q\right)q^{k}a_{k}(z,q)\\
&=\left(1-z/q\right)\left(\sum_{i=1}^{k+1}f(i)-\sum_{j=1}^{k}a_j(z,q) q^j \right).
\end{align*}
This proves (2).
\end{proof}

\begin{theorem}\label{hwalitheorem}
Let $a_n(z,q)$ be defined in \eqref{anzq}.
Then there exist rational coefficients $h_{j}$ such that 
\begin{align}
\lim_{n\to \infty}\left\{\left(1-z/q\right)\sum_{j=1}^{n}f(j)-a_n(z,q)\right\}=\sum_{j=1}^{\infty}h_jM_{j};
\end{align}
where 
\begin{align*}
for~j\geq 2,~ h_j =\sum_{i \geq j-1}(-1)^{i-j+1}\binom{i-1}{j-2}i!\sum_{k \geq i}c_k \tilde{s}(k,i);~~ h_1 =c_0,
\end{align*}
and $\tilde{s}(k,i)$ are Stirling numbers of the second kind.
\end{theorem}

\begin{proof}
Define
\begin{align} 
A(z, \alpha , q):&=\sum_{n=1}^{\infty}a_n(z,q)\alpha^n,\label{azalpha}\\
F(\a):&= \sum_{n=1}^{\infty}f(n)\a^n.
\end{align}
By using \eqref{anzq}, we see that
\begin{align*}
A(z, \alpha , q)&=\sum_{n=1}^{\infty}a_n(z,q)\alpha^n\\
&=\sum_{n=1}^{\infty}\left(\left(1-z/q \right)f(n)+\left\{1-\left(1-z/q \right)q^{n-1}\right\}a_{n-1}(z,q)\right)\alpha^n\\
&=\left(1-z/q \right)F(n)+\a A(z,\a,q)-\a \left(1-z/q \right)  A(z,\a q,q).
\end{align*}
Therefore,
\begin{align*}
A(z, \alpha , q)&=\frac{\left(1-z/q \right)}{(1-\a)}F(n) - \frac{\left(1-z/q \right)\a}{(1-\a)}A(z, \a q, q).
\end{align*}
Solving the above recursive relation leads to 
\begin{align*}
A(z, \alpha , q)=\sum_{n=1}^{\infty}\frac{(-1)^{n-1}\left(z/\a q^{n-1};q\right)_n F(\a q^{n-1})q^{ \frac{(n-1)(n-2)}{2}}\a^{n-1}}{(\a;q)_n},
\end{align*}
thus, at $\a =q$,
\begin{align}\label{nj}
A(z, q , q)=\sum_{n=1}^{\infty}\frac{(-1)^{n-1}\left(z/q^{n};q\right)_n F(q^{n})q^{ \frac{n(n-1)}{2}}}{(q;q)_n}.
\end{align}
Upon using the following simple observation
$$(-1)^n\left(z/q^{n};q\right)_n q^{\frac{n(n-1)}{2}} = (q/z;q )_n (z/q)^{n}$$
in \eqref{nj}, we have 
\begin{align}\label{Azqq}
A(z, q , q)=-\sum_{n=1}^{\infty}\frac{F(q^{n})\left(q/z;q \right)_n}{(q;q)_n}\left(z/q\right)^{n}.
\end{align}
Invoke \eqref{angrepreqn2} and use \eqref{azalpha} to deduce that
\begin{align}\label{limx}
\lim_{n\to \infty}\left(\left(1-z/q\right)\sum_{j=1}^{n}f(j)-a_{n}(z,q)\right)=\left(1-z/q\right) A(z,q,q).
\end{align}
Now let us state another form of a generating function of $f(n)$, as given in \cite{ACS}, namely, 
\begin{align}\label{f(n)}
F(\a) =\sum_{m\geq 0}\sum_{k \geq m}c_k \tilde{s}(k,m) m! \frac{\a^m}{(1-\a)^{m+1}}-c_0,
\end{align} 
then, 
\begin{align}\label{F(q^n)}
F(q^n) =\sum_{m\geq 1}d_m\sum_{j=0}^{m-1}\binom{m-1}{j}(-1)^j\frac{q^n}{q^{m+1-j}}+c_0\frac{q^n}{1-q^n},
\end{align}
where
\begin{align*}
d_m = \sum_{k \geq m}c_k \tilde{s}(k,m)m!,~ \textup{for}~ m \geq 1,~~ d_0 =c_0.
\end{align*}
From equation \eqref{Azqq}, \eqref{limx} and \eqref{F(q^n)}, 
\begin{align}
\lim_{n\to \infty}\left(\left(1-z/q\right)\sum_{j=1}^{n}f(j)-a_{n}(z,q)\right)&=\left(1-z/q\right) A(z,q,q)\nonumber\\
&=-(1-z/q)\sum_{n=1}^{\infty}\frac{F(q^{n})\left(\frac{q}{z};q \right)_n}{(q)_n}\left(\frac{z}{q}\right)^{n}\label{liman} \\
&=-(1-z/q)\Bigg\{\sum_{m \geq 1}d_m\sum_{j=0}^{m-1}e_{m,j}\sum_{n =1 }^{\i}\frac{(q/z;q)_nz^n}{(q;q)_n(1-q^n)^{m+1-j}}\nonumber \\
&\qquad\qquad \qquad \qquad \qquad \qquad+c_0\sum_{n =1 }^{\i}\frac{(q/z;q)_nz^n}{(q;q)_n(1-q^n)}\Bigg\}.\nonumber
\end{align}
Employ Theorem \ref{maintheorem2} in the above equation to obtain 
\begin{align*}
\lim_{n\to \infty}\left(\left(1-z/q\right)\sum_{j=1}^{n}f(j)-a_{n}(z,q)\right)&=-(1-z/q)\left\{ -\sum_{m \geq 1}d_m\sum_{j=0}^{m-1}e_{m,j}M_{m+1-j}-c_0M_1 \right\}\\
&=(1-z/q)\sum_{j=1}^{\i}h_jM_j.
\end{align*}
This proves the theorem.
\end{proof}

\begin{remark}
It is easy to see that upon letting $z\to0$ in Theorem \ref{hwalitheorem}, we obtain \cite[Theorem 3.1]{ACS}.
\end{remark}

%

As a special case of Theorem \ref{hwalitheorem} we get the following interesting result.
\begin{corollary}\label{bthm}
Let $b\in\mathbb{C}\backslash\{1\}$, and $a_{n,z}(q)$ is the sequence of polynomials in $q$ defined recursively for $n\in\mathbb{N}$ by
\begin{align}
a_{n}(z,q):= \left(1-z/q\right)b^n+\left\{1-\left(1-z/q\right)q^{n-1}\right\}a_{n-1}(z,q),\qquad a_0(q)=0.
\end{align}
Then, for $|q|<\min(|b|^{-1},1)$, 
\begin{align}\label{bthmeqn}
\lim_{n\to\infty}\left((1-z/q)\frac{b-b^{n+1}}{1-b}-a_{n,z}(q)\right)=-b(1-z/q)\sum_{m=0}^\infty \frac{(q/b;q)_mb^m}{(q;q)_m}\left(\frac{zq^m}{1-zq^m}-\frac{q^{m+1}}{1-q^{m+1}}\right).
\end{align}
\end{corollary}

\begin{proof}
Define
\begin{align}\label{a}
a:=1-z/q,
\end{align}
Lemma \ref{anqrepr}~(1) with $f(n)=b^n$ implies that
\begin{align}\label{b11}
a_{n}(z,q)&=a\sum_{i=0}^{n-1}b^{n-i}\prod_{j=1}^i\left(1-aq^{n-j}\right).
\end{align}
Multiply and divide \eqref{b11} by $(aq;q)_{n-i-1}$ to see that
\begin{align}
a_{n}(z,q)&=a(aq;q)_{n-1}\sum_{i=0}^{n-1}\frac{b^{n-i}}{(aq;q)_{n-i-1}}.\nonumber\\
\end{align}
Replace $n-i-1$ by $j$ in the above equation and use the value of $a$ from \eqref{a} to get
\begin{align}
a_{n}(z,q)&=\left(1-z/q\right)\left(\left(1-z/q\right)q;q\right)_{n-1}\sum_{j=0}^{n-1}\frac{b^{j+1}}{\left(\left(1-z/q\right)q;q\right)_j}.
\end{align}
Let 
\begin{align}
F(x):=\sum_{n=1}^\infty b^nx^n.
\end{align}
Then 
\begin{align}\label{b12}
F(q^n)=\frac{bq^n}{1-bq^n},
\end{align}
for $|q|<|b|^{-1}$ and $n\in\mathbb{N}$. From \eqref{Azqq} and \eqref{b12}, 
\begin{align}\label{b14}
A(z,q,q)=-b\sum_{n=1}^\infty \frac{(q/z;q)_n}{(1-bq^n)(q;q)_n}z^n.
\end{align}
Equation \eqref{limx} with $f(n)=b^n$ implies 
\begin{align}\label{b15}
\lim_{n\to\infty}\left((1-z/q)\frac{b-b^{n+1}}{1-b}-a_{n,z}(q)\right)=(1-z/q)A(a,q,q).
\end{align}
From \eqref{b14} and \eqref{b15}, 
\begin{align}\label{b16}
\lim_{n\to\infty}\left((1-z/q)\frac{b-b^{n+1}}{1-b}-a_{n,z}(q)\right)=-b(1-z/q)\sum_{n=1}^\infty \frac{(q/z;q)_n}{(1-bq^n)(q;q)_n}z^n.
\end{align}
Substitute \eqref{DM} with $a=z,\ b=q$ and then $c=b$ in \eqref{b16} to get
\begin{align}
\lim_{n\to\infty}\left((1-z/q)\frac{b-b^{n+1}}{1-b}-a_{n,z}(q)\right)=-b(1-z/q)\sum_{m=0}^\infty \frac{(q/b;q)_mb^m}{(q;q)_m}\left(\frac{zq^m}{1-zq^m}-\frac{q^{m+1}}{1-q^{m+1}}\right).
\end{align}
This proves the corollary.
\end{proof}

\begin{remark}
Upon letting $z\to0$ in Corollary \ref{bthm}, one get \cite[Theorem 1.2]{bring}.
\end{remark}

We end this section with proving Theorem \ref{bringperiod} and its corollary.

\begin{proof}[Theorem \textup{\ref{bringperiod}}][]
We will use the following equivalent representation obtained in \cite[Equation (3.8)]{bring}, 
\begin{align}\label{fn}
f(n) =\sum_{j=1}^{N}\frac{f(j)}{N}\sum_{k=0}^{N-1}\zeta^{(n-j)k}_{N},
\end{align}
since, 
\begin{align*}
\sum_{k=0}^{N-1}\zeta^{(n-j)k}_{N} =
\left\{
	\begin{array}{ll}
		N  & n \equiv j~ (mod~ N), \\
		0 & otherwise,
	\end{array}
\right.
\end{align*}
then it is easy to conclude that \eqref{fn} is a periodic function of period $N$. Also for $|\a|<1$,
\begin{align}\label{fxnm}
F(x) =\sum_{k=0}^{N-1}\frac{c_k x}{1-\zeta^k_N x}.
\end{align}
Thus, from  \eqref{liman} and \eqref{fxnm}, 
\begin{align}\label{F(q)}
\lim_{n\to \infty}\left(\left(1-z/q\right)\sum_{j=1}^{n}f(j)-a_{n}(z,q)\right)
&=-(1-z/q)\sum_{n=1}^{\i}\sum_{k=0}^{N-1}\frac{(q/z;q)_n}{(q;q)_n}\frac{c_k q^n}{1-\zeta^k_N q^n}(z/q)^n \nonumber \\
&=-(1-z/q)\left(
c_0\sum_{n=1}^{\i}\frac{(q/z;q)_nz^n}{(1-q^n) (q;q)_n}+ \sum_{k=1}^{N-1}c_k\sum_{n=1}^{\i}\frac{(q/z;q)_nz^n}{(1-\zeta^k_Nq^n) (q;q)_n}\right).
\end{align}
Employing \eqref{maintheorem2eqn} with $k=1$ in the first sum on the right-hand side of\eqref{F(q)} and using the $q$-Gauss summation formula \eqref{qgauss} with $a=q/z, ~b= \zeta^k_N,~ and ~ c=q\zeta^k_N$ in the second sum on the right-hand side of \eqref{F(q)}, we deduce that 
\begin{align}
\lim_{n\to \infty}\left(\left(1-z/q\right)\sum_{j=1}^{n}f(j)-a_{n}(z,q)\right)&=-(1-z/q)\left(
-c_0\mathfrak{S}_{0,z}(q)+ \sum_{k=1}^{N-1}c_k\left(\frac{(z\zeta^k_N;q)_\i(cq;q)_\i}{(\zeta^k_N;q)_\i(z;q)_\i}-\frac{1}{1-\zeta^k_N} \right)\right)\nonumber \\ 
&=(1-z/q)\left(
c_0\mathfrak{S}_{0,z}(q)+\sum_{k=1}^{N-1}\frac{c_k}{1-\zeta^k_N} - \frac{(q;q)_\i}{(z;q)_\i}\sum_{k=1}^{N-1}c_k\frac{(z\zeta^k_N;q)_\i}{(q\zeta^k_N;q)_\i}\right).\nonumber
\end{align}
This concludes the proof.
\end{proof}

Our next next result is a generalization of \cite[Corollary 1.4]{bring}.
\begin{corollary}\label{gen1.4}
Let $f(n)$ be a periodic sequence with period $N$ and $a_n(z,q)$ is the sequence such that
\begin{align*}
a_{n}(z,q):= (1-z/q)f(n)+\left\{1-(1-z/q)q^{n-1}\right\}a_{n-1}(z,q),\qquad a_0(z,q)=0.
\end{align*}
Then for $|q|<1$, we have
\begin{align}\label{gen1.4eqn}
\lim_{n\to \infty}\left((1-z/q)\sum_{j=1}^{n}f(j)-a_{n}(z,q)\right)=(1-z/q)\frac{(q;q)_{\infty}}{(z;q)_{\infty}}\sum_{n=0}^\infty\frac{(z/q;q)_nq^n}{(q;q)_n}\sum_{j=1}^Nf(j)\left\lceil\frac{n+1-j}{N}\right\rceil.
\end{align}
\end{corollary}

\begin{proof}
 We first simplify the terms of \eqref{bringperiodeqn}. Note that by using the definition of $c_k$ given in Theorem \ref{bringperiod}, we have
\begin{align}\label{vx1}
\sum_{k=1}^{N-1}\frac{c_k(z\zeta_N^k;q)_{\infty}}{(\zeta_N^k;q)_{\infty}}&=\frac{1}{N}\sum_{k=1}^{N-1}\frac{(z\zeta_N^k;q)_{\infty}}{(\zeta_N^k;q)_{\infty}}\sum_{j=1}^Nf(j)\zeta_N^{(1-j)k}\nonumber\\
&=\frac{1}{N}\sum_{j=1}^Nf(j)\sum_{k=1}^{N-1}\frac{(z\zeta_N^k;q)_{\infty}\zeta_N^{(1-j)k}}{(\zeta_N^k;q)_{\infty}}\nonumber\\
&=\frac{1}{N}\sum_{j=1}^Nf(j)\sum_{k=1}^{N-1}\frac{(z\zeta_N^k;q)_{\infty}\zeta_N^{(1-j)k}}{(1-\zeta_N^k)(\zeta_N^kq;q)_{\infty}}.
\end{align}
Use \eqref{analogqbinom} with $a=z,\ b=q$ and $y=\zeta_N^k$ to represent $(z\zeta_N^k;q)_{\infty}/(\zeta_N^kq;q)_{\infty}$ as a series and then substitute it in \eqref{vx1} to deduce that
\begin{align}\label{vx2}
\sum_{k=1}^{N-1}\frac{c_k(z\zeta_N^k;q)_{\infty}}{(\zeta_N^k;q)_{\infty}}&=\frac{1}{N}\sum_{j=1}^Nf(j)\sum_{k=1}^{N-1}\sum_{n=0}^{\infty}\frac{(z/q;q)_nq^n}{(q;q)_n}\frac{\zeta_N^{nk+(1-j)k}}{(1-\zeta_N^k)}\nonumber\\
&=\frac{1}{N}\sum_{j=1}^Nf(j)\sum_{n=0}^{\infty}\frac{(z/q;q)_nq^n}{(q;q)_n}\sum_{k=1}^{N-1}\frac{\zeta_N^{(n+1-j)k}}{(1-\zeta_N^k)}.
\end{align}
From \cite[Lemma 2.5]{bring}, for $j\in\mathbb{Z}$, we have
\begin{align}
\sum_{k=1}^{N-1}\frac{\zeta_N^{jk}}{1-\zeta_N^k}=\frac{N-1}{2}+j-N\left\lceil\frac{j}{N}\right\rceil.
\end{align}
By using the above equation with $j$ replaced by $n+1-j$ in \eqref{vx2}, we get
\begin{align}\label{vx3}
\sum_{k=1}^{N-1}\frac{c_k(z\zeta_N^k;q)_{\infty}}{(\zeta_N^k;q)_{\infty}}&=\frac{1}{N}\sum_{j=1}^Nf(j)\sum_{n=0}^{\infty}\frac{(z/q;q)_nq^n}{(q;q)_n}\left(\frac{N-1}{2}+n+1-j-N\left\lceil\frac{n+1-j}{N}\right\rceil\right)\nonumber\\
&=\frac{1}{N}\sum_{j=1}^Nf(j)\left(\frac{N+1}{2}-j\right)\sum_{n=0}^{\infty}\frac{(z/q;q)_nq^n}{(q;q)_n}+\frac{1}{N}\sum_{j=1}^Nf(j)\sum_{n=0}^{\infty}\frac{n(z/q;q)_nq^n}{(q;q)_n}\nonumber\\
&\qquad-\sum_{n=0}^{\infty}\frac{(z/q;q)_nq^n}{(q;q)_n}\sum_{j=1}^Nf(j)\left\lceil\frac{n+1-j}{N}\right\rceil.
\end{align}
Invoke \eqref{qanalog} and Theorem \ref{maintheorem1} with $k=1$ to evaluate the first sum and second sum respectively on the right-hand side of \eqref{vx3} to arrive at
\begin{align}\label{1}
\sum_{k=1}^{N-1}\frac{c_k(z\zeta_N^k;q)_{\infty}}{(\zeta_N^k;q)_{\infty}}&=\frac{(z;q)_{\infty}}{N(q;q)_{\infty}}\sum_{j=1}^Nf(j)\left(\frac{N+1}{2}-j\right)+\frac{(z;q)_{\infty}\mathfrak{S}_{0}(z,q)}{N(q;q)_{\infty}}\sum_{j=1}^Nf(j)\nonumber\\
&\qquad-\sum_{n=0}^{\infty}\frac{(z/q;q)_nq^n}{(q;q)_n}\sum_{j=1}^Nf(j)\left\lceil\frac{n+1-j}{N}\right\rceil.
\end{align}
From \cite[p.~6]{bring}, we have
\begin{align}\label{2}
\sum_{k=1}^{N-1}\frac{c_k}{1-\zeta_N^k}=\frac{1}{N}\sum_{j=1}^{N}f(j)\left(\frac{N+1}{2}-j\right).
\end{align}
Finally, substitute values from \eqref{1} and \eqref{2} in \eqref{bringperiodeqn} to arrive at \eqref{gen1.4eqn}.
\end{proof}

Next, we prove Corollary \ref{gk10} by employing Theorem \ref{bringperiod} and Corollary \ref{gen1.4}.
\begin{proof}[Corollary \textup{\ref{gk10}}][]
Let $f(n)=(-1)^n$. Then it is easy to see that $c_0=0$ and $c_1=-1$. Also $f(n)$ is a periodic sequence with period $N=2$.  Therefore upon invoking Theorem \ref{bringperiod}, we get
\begin{align}\label{zx1}
\lim_{n\to \infty}\left((1-z/q)\sum_{j=1}^{n}(-1)^j-a_{n}(z,q)\right)&=(1-z/q)\left(-\frac{1}{2}+\frac{1}{2}\frac{(q;q)_{\infty}(-z;q)_{\infty}}{(-q;q)_{\infty}(z;q)_{\infty}}\right).
\end{align}
An application of Theorem \ref{gen1.4} implies that
\begin{align}\label{zx2}
\lim_{n\to \infty}\left((1-z/q)\sum_{j=1}^{n}(-1)^j-a_{n}(z,q)\right)&=(1-z/q)\frac{(q;q)_{\infty}}{(z;q)_{\infty}}\sum_{n=0}^\infty\frac{(z/q;q)_{n}q^{n}}{(q;q)_{n}}\left(-\left\lceil\frac{n}{2}\right\rceil+\left\lceil\frac{n-1}{2}\right\rceil\right)\nonumber\\
&=-(1-z/q)\frac{(q;q)_{\infty}}{(z;q)_{\infty}}\sum_{n=0}^\infty\frac{(z/q;q)_{2n+1}q^{2n+1}}{(q;q)_{2n+1}}.
\end{align}
The result now follows upon using \eqref{zx1} and \eqref{zx2}.
\end{proof}

\section{Theory of the generalized divisor function $\sigma_{s,z}(n)$}\label{pgdf}

In this section we develop the theory of the generalized divisor function $\sigma_{s,z}(n)$. We begin with presenting the proof of Theorem \ref{sumddn}. For that we need to use the Dirichlet product of arithmetical functions $f(n)$ and $g(n)$ given by \cite[p.~29]{apostal}
\begin{align}
(f\ast g)(n):=\sum_{d|n}f(d)g\left(\frac{n}{d}\right).
\end{align}

\begin{proof}[Theorem \textup{\ref{sumddn}}][]
Let us define, for $n\in\mathbb{N}$,
\begin{align}
N_z^s(n)&:=n^sz^n,\label{nxiz}\\
N(n)&:=n,\label{nz}
\end{align}
and 
\begin{align}\label{un}
U(n):=1.
\end{align}
Note that
\begin{align}\label{nxisi}
\left(N_z^s\ast u\right)(n)&=\sum_{d|n}N_z^s(d)u\left(\frac{n}{d}\right)\nonumber\\
&=\sum_{d|n} d^sz^d\nonumber\\
&=\sigma_{s,z}(n).
\end{align}
Use \eqref{nxisi} in the second step below so that
\begin{align}\label{sumdn}
\sum_{d|n}\varphi(d)\sigma_{s,z}\left(\frac{n}{d}\right)&=\left(\varphi\ast\sigma_{s,z}\right)(n)\nonumber\\
&=\left(\varphi\ast\left(N_z^s\ast u\right)\right)(n)\nonumber\\
&=\left(\varphi\ast\left(u\ast N_z^s\right)\right)(n)\nonumber\\
&=\left(\left(\varphi\ast u\right)\ast N_z^s\right)(n),
\end{align}
where we used the commutative and associative property of the Dirichlet product. In the notation of Dirichlet product and \eqref{nz}, Theorem 2.2 of \cite[p.~26]{apostal} implies that
\begin{align}\label{apthm}
(\varphi\ast u)(n)=N(n).
\end{align}
From \eqref{sumdn} and \eqref{apthm}, 
\begin{align}
\sum_{d|n}\varphi(d)\sigma_{s,z}\left(\frac{n}{d}\right)&=\left(N\ast N_z^s\right)(n)\nonumber\\
&=\left(N_z^s\ast N \right)(n)\nonumber\\
&=\sum_{d|n}N_z^s(d)N\left(\frac{n}{d}\right)\nonumber\\
&=\sum_{d|n}d^sz^d \frac{n}{d}\nonumber\\
&=n\sum_{d|n}d^{s-1}z^d\nonumber
\end{align}
Finally upon using the definition of $\sigma_{s,z}(n)$ in the above equation, we arrive at \eqref{sumddneqn}.
\end{proof}

\begin{proof}[Corollary \textup{\ref{spz10}}][]
Let $z=1$ in Theorem \ref{sumddn} and use the fact that $\sigma_{s-1}(n)=n^{s-1}\sigma_{1-s}(n)$ and then in the resultant expression replace $s$ by $1-s$ to arrive at \eqref{spz10eqn}.
\end{proof}

Next we obtain the Dirichlet series for $\sigma_{s,z}(n)$.
\begin{proof}[Theorem \textup{\ref{dirichlet}}][]
Upon using the definition \eqref{newdivisor} of $\sigma_{z,\xi}(n)$, we see that
\begin{align}
\sum_{n=1}^\infty\frac{\sigma_{s,z}(n)}{n^{\alpha}}&=\sum_{n=1}^\infty\frac{1}{n^{\alpha}}\sum_{d|n}d^sz^d\nonumber\\
&=\sum_{d=1}^\infty\sum_{m=1}^\infty\frac{d^sz^d}{(md)^{\alpha}}\nonumber\\
&=\sum_{d=1}^\infty\frac{z^d}{d^{\alpha-s}}\sum_{m=1}^\infty\frac{1}{m^{\alpha}}\nonumber\\
&=\mathrm{Li}_{\alpha-s}(z)\zeta(\alpha),
\end{align}
where we used the definitions of $\zeta(\alpha)$ and $\mathrm{Li}_{\alpha-z}(z)$ for Re$(\alpha)>1$ and for all $\alpha\in\mathbb{C}$, respectively.
\end{proof}

\begin{proof}[Corollary \textup{\ref{zetadircor}}][]
From \eqref{poly}, it is clear to see that, for Re$(\alpha)>1+\mathrm{Re}(s)$,
\begin{align}\label{diri1l}
\mathrm{Li}_{\alpha-s}(1)=\zeta(\alpha-s).
\end{align}
Let $z=1$ in \eqref{dirichleteqn} and then use \eqref{diri1l} to arrive at \eqref{zetadir}.
\end{proof}

We note down the Euler's summation formula \cite[Theorem 3.1]{apostal} which is crucial to proving Lemma \ref{ls}.
\begin{theorem}\label{esf}
If $f$ has a continuous derivative $f'$ on the interval $[y, x]$, where $0<y<x$, then
\begin{align}
\sum_{y<n\leq x}f(n)=\int_y^xf(t)\ dt+\int_y^x(t-\lfloor t\rfloor)f'(t)\ dt+f(x)(x-\lfloor x\rfloor)-f(y)(y-\lfloor y\rfloor).
\end{align}
\end{theorem}

We first prove the following lemma which will be employed later. 
\begin{lemma}\label{ls}
Let $\alpha>0$ and $0<z\leq1$. Then, for $x\geq2$, we have
\begin{align}\label{lseqn}
\sum_{n\leq x}\frac{z^n}{n^{\alpha}}=-x^{1-s}E_{\alpha}(-x\log(z))+\mathrm{Li}_{\alpha}(z)+O\left(x^{-\alpha}\right).
\end{align}
\end{lemma}

\begin{proof}
Let $f(t)=z^t/t^{\alpha}$ and $y=1$ in Theorem \ref{esf} to get
\begin{align}\label{ls1}
\sum_{n\leq x}\frac{z^n}{n^{\alpha}}&=z+\int_1^x \frac{z^t}{t^{\alpha}}\ dt+\int_1^x(t-\lfloor t\rfloor)\left(-\alpha\frac{z^t}{t^{{\alpha}+1}}+\frac{z^t\log(z)}{t^{\alpha}}\right)\ dt-(x-\lfloor x\rfloor)\frac{z^x}{x^{\alpha}} \nonumber\\
&=E_{\alpha}(-\log(z))-x^{1-{\alpha}}E_{\alpha}(-x\log(z))+z-{\alpha}\int_1^x \frac{(t-\lfloor t\rfloor)z^t}{t^{{\alpha}+1}}dt+\log(z)\int_1^x \frac{(t-\lfloor t\rfloor)z^t}{t^{{\alpha}}}dt+O(x^{-{\alpha}}).
\end{align}
Note that for $0<z\leq1$,
\begin{align}\label{ls2}
\left|\int_x^\infty \frac{(t-\lfloor t\rfloor)z^t}{t^{\alpha+1}}dt\right|&\leq\int_x^\infty \frac{|z|^t}{t^{{\alpha}+1}}\ dt\nonumber\\
&\leq\int_x^\infty \frac{1}{t^{{\alpha}+1}}\ dt\nonumber\\
&=O(x^{-{\alpha}}),
\end{align}
and
\begin{align}\label{ls3}
\left|\int_x^\infty \frac{(t-\lfloor t\rfloor)z^t}{t^{{\alpha}}}dt\right|&\leq\int_x^\infty \frac{|z|^t}{t^{{\alpha}}}\ dt\nonumber\\
&\leq \int_x^\infty \frac{1}{t^{{\alpha}}}\ dt\nonumber\\
&=O(x^{-{\alpha}}).
\end{align}
From \eqref{ls1}, \eqref{ls2} and \eqref{ls3},
\begin{align}\label{ls4}
\sum_{n\leq x}\frac{z^n}{n^{\alpha}}&=E_{\alpha}(-\log(z))-x^{1-{\alpha}}E_{\alpha}(-x\log(z))+z-{\alpha}\int_1^\infty \frac{(t-\lfloor t\rfloor)z^t}{t^{{\alpha}+1}}dt+\log(z)\int_1^\infty \frac{(t-\lfloor t\rfloor)z^t}{t^{{\alpha}}}dt+O(x^{-{\alpha}})\nonumber\\
&=E_{\alpha}(-\log(z))-x^{1-{\alpha}}E_{\alpha}(-x\log(z))+z-\int_1^\infty \frac{(t-\lfloor t\rfloor)}{t^{{\alpha}}}\left(\frac{{\alpha}}{t}-\log(z)\right)z^t\ dt+O(x^{-{\alpha}})\nonumber\\
&=-x^{1-{\alpha}}E_{\alpha}(-x\log(z))+C_{z}({\alpha})+O(x^{-{\alpha}}),
\end{align}
where,
\begin{align*}
C_{z}({\alpha})=E_{\alpha}(-\log(z))+z-\int_1^\infty \frac{(t-\lfloor t\rfloor)}{t^{{\alpha}}}\left(\frac{{\alpha}}{t}-\log(z)\right)z^t\ dt.
\end{align*}
Note that upon taking $x\to\infty$ in \eqref{ls4}, we get
\begin{align*}
\mathrm{Li}_{\alpha}(z)=C_{z}({\alpha}),
\end{align*}
therefore 
\begin{align}\label{csxi}
C_{z}(\alpha)=\mathrm{Li}_{\alpha}(z).
\end{align}
Hence from \eqref{ls4} and \eqref{csxi}, we have
\begin{align*}
\sum_{n\leq x}\frac{z^n}{n^{\alpha}}&=-x^{1-{\alpha}}E_{\alpha}(-x\log(z))+\mathrm{Li}_{\alpha}(z)+O(x^{-{\alpha}}).
\end{align*}
This proves \eqref{lseqn}.
\end{proof}

As a special case of Lemma \ref{ls}, we get \cite[p.~55, Theorem 3.2(b)]{apostal}.
\begin{corollary}
For ${\alpha}>1$, we have
\begin{align}\label{zetasf}
\sum_{n\leq x}\frac{1}{n^{\alpha}}=\frac{x^{1-{\alpha}}}{1-{\alpha}}+\zeta({\alpha})+O(x^{-{\alpha}}).
\end{align}
\end{corollary}
\begin{proof}
Let $z=1$ in Lemma \ref{ls} and use the fact that $\mathrm{Li}_{\alpha}(1)=\zeta({\alpha})$ for $\alpha=1$ to see that
\begin{align}\label{zetasf1}
\sum_{n\leq x}\frac{1}{n^{\alpha}}&=-x^{1-{\alpha}}E_{\alpha}(0)+\zeta({\alpha})+O(x^{-{\alpha}}).
\end{align}
From \cite[p.~229, Formula 5.1.23]{handbook}, for $\nu>1$
\begin{align}\label{en0}
E_{\nu}(0)=\frac{1}{\nu-1}.
\end{align}
Let $\nu={\alpha}$ in \eqref{en0}, then for ${\alpha}>1$, 
\begin{align}\label{zetasf2}
E_{\alpha}(0)&=-\frac{1}{1-{\alpha}}.
\end{align}
Substitute \eqref{zetasf2} in \eqref{zetasf1} to arrive at \eqref{zetasf}.
\end{proof}

Our next theorem gives Theorem \ref{s=0caseofgdf} very easily.
\begin{theorem}\label{logsf}
Let ${\alpha}>0,\ {\alpha}\neq1,\ s>0$ and $0<z<1$. Then, we have
\begin{align}\label{logsfeqn}
\sum_{n\leq x}\frac{\sigma_{s,z}(n)}{n^{\alpha}}&=-\frac{x^{1+s-{\alpha}}}{1-{\alpha}}E_{1-s}(-x\log(z))-x^{1+s-{\alpha}}\zeta({\alpha})E_{{\alpha}-s}(-x\log(z))+\zeta({\alpha})\mathrm{Li}_{{\alpha}-s}(z)\nonumber\\
&\qquad+\frac{x^{1-{\alpha}}}{1-{\alpha}}\mathrm{Li}_{1-s}(z)+O(x^\beta),
\end{align}
where $\beta=\mathrm{max}\left\{-{\alpha},x^{1+s-{\alpha}}E_{-s}(-x\log(z))\right\}$.
\end{theorem}

\begin{proof}
Invoking the definition of $ \sigma_{s,z}(n)$ to see that
\begin{align}\label{dirs1}
\sum_{n\leq x}\frac{\sigma_{s,z}(n)}{n^{\alpha}}&=\sum_{n\leq x}\frac{1}{n^{\alpha}}\sum_{d|n}d^sz^d\nonumber\\
&=\sum_{d\leq x}\sum_{q\leq x/d}\frac{d^sz^d}{q^{\alpha}d^{\alpha}} \nonumber\\
&=\sum_{d\leq x}d^{s-{\alpha}}z^d\sum_{q\leq x/d}\frac{1}{q^{\alpha}}.
\end{align}
From \cite[p.~55, Theorem 3.2(b)]{apostal}, for ${\alpha}>0, {\alpha}\neq1$, we have
\begin{align}\label{apostal}
\sum_{n\leq x}\frac{1}{n^{\alpha}}=\frac{x^{1-{\alpha}}}{1-{\alpha}}+\zeta({\alpha})+O(x^{-{\alpha}}).
\end{align}
Use \eqref{apostal} in \eqref{dirs1} to get, for $\alpha>0,\ \alpha\neq1$,
\begin{align}\label{dirs3}
\sum_{n\leq x}\frac{\sigma_{s,z}(n)}{n^{\alpha}}&=\sum_{d\leq x}d^{s-{\alpha}}z^d\left(\frac{(x/d)^{1-{\alpha}}}{1-{\alpha}}+\zeta({\alpha})+O\left(\left(x/d\right)^{-{\alpha}}\right)\right)\nonumber\\
&=\frac{x^{1-{\alpha}}}{1-{\alpha}}\sum_{d\leq x}\frac{z^d}{d^{1-s}}+\zeta({\alpha})\sum_{d\leq x}\frac{z^d}{d^{{\alpha}-s}}+O\left(x^{-{\alpha}}\sum_{d\leq x}\frac{z^d}{d^{-s}}\right).
\end{align}
Let ${\alpha}=1-s$ in Lemma \ref{ls} to get, for $s<1$ and $0<z\leq1$,
\begin{align}\label{dirs2}
\sum_{d\leq x}\frac{z^d}{d^{1-s}}=-x^{s}E_{1-s}(-x\log(z))+\mathrm{Li}_{1-s}(z)+O\left(x^{s-1}\right).
\end{align}
Upon Invoking Lemma \ref{ls} with replacing ${\alpha}$ by ${\alpha}-s$, we get, for $\alpha>s$ and $0<z\leq1$,
\begin{align}\label{dirs41}
\sum_{d\leq x}\frac{z^d}{d^{{\alpha}-s}}=-x^{1+s-{\alpha}}E_{{\alpha}-s}(-x\log(z))+\mathrm{Li}_{{\alpha}-s}(z)+O\left(x^{{\alpha}-s}\right).
\end{align}
We employ Lemma \ref{ls} again with ${\alpha}=-s$ so that, for $s<0$ and $0<z\leq1$,
\begin{align}\label{dirs5}
\sum_{d\leq x}\frac{z^d}{d^{-s}}=-x^{1+s}E_{-s}(-x\log(z))+\mathrm{Li}_{-s}(z)+O\left(x^{s}\right)
\end{align}
Substitute values from \eqref{dirs2}, \eqref{dirs41} and \eqref{dirs5} in \eqref{dirs3} to arrive at
\begin{align}
\sum_{n\leq x}\frac{\sigma_{s,z}(n)}{n^{\alpha}}&=-\frac{x^{1+s-{\alpha}}}{1-{\alpha}}E_{1-s}(-x\log(z))+\frac{x^{1-{\alpha}}}{1-{\alpha}}\mathrm{Li}_{1-s}(z)+O\left(x^{{\alpha}-s}\right)-x^{1+s-{\alpha}}\zeta({\alpha})E_{{\alpha}-s}(-x\log(z))\nonumber\\
&\quad+\zeta({\alpha})\mathrm{Li}_{{\alpha}-s}(z)+O\left(x^{{\alpha}-s}\right)+O\left(x^{1+s-{\alpha}}E_{-s}(-x\log(z))+x^{-{\alpha}}+x^{{\alpha}-s}\right)\nonumber\\
&=-\frac{x^{1+s-{\alpha}}}{1-{\alpha}}E_{1-s}(-x\log(z))+\frac{x^{1-{\alpha}}}{1-{\alpha}}\mathrm{Li}_{1-s}(z)-x^{1+s-{\alpha}}\zeta({\alpha})E_{{\alpha}-s}(-x\log(z))+\zeta({\alpha})\mathrm{Li}_{{\alpha}-s}(z)\nonumber\\
&\quad+O\left(x^{\beta}\right),
\end{align}
where $\beta:=\mathrm{Re}\left\{x^{-{\alpha}},x^{1+s-{\alpha}}E_{-s}(-x\log(z))\right\}$. This proves the theorem.
\end{proof}

Theorem \ref{s=0caseofgdf} follows very easily by invoking Theorem \ref{logsf}.
\begin{proof}[Theorem \textup{\ref{s=0caseofgdf}}][]
Let $\alpha\to0$ in \eqref{logsfeqn} and use the fact that $\zeta(0)=-1/2$ to arrive at Theorem \ref{s=0caseofgdf}.
\end{proof}

We also get the following result for $\sigma_{-s}(n)$ \cite[p.~61, Theorem 3.6]{apostal} from Theorem \ref{s=0caseofgdf}.
\begin{corollary}\label{s=0xi=1caseofgdf}
Let $s>1$. Then if $x>1$ we have
\begin{align}\label{s=0xi=1caseofgdfeqn}
\sum_{n\leq x}\sigma_{-s}(n)&=x\zeta(1+s)+O(1).
\end{align}
\end{corollary}

\begin{proof}
Upon invoking \eqref{en0}, for $s<0$, we get
\begin{align}\label{01}
E_{1-s}(0)=-\frac{1}{s}.
\end{align}
Also,  for $s<-1$,
\begin{align}\label{02}
E_{-s}(0)=-\frac{1}{s+1}.
\end{align}
Let $z=1$ in Theorem \ref{s=0caseofgdf} and then in the resultant expression use \eqref{01} and \eqref{02} so that, for $s<-1$,
\begin{align*}
\sum_{n\leq x}\sigma_{s}(n)&=\frac{x^{1+s}}{s}+\frac{x^{1+s}}{2(1+s)}-\frac{1}{2}\zeta(-s)+x\zeta(1-s)+O(1).
\end{align*} 
Upon replacing $s$ by $-s$ in the above equation then for $s>1$, we get
\begin{align}\label{q1}
\sum_{n\leq x}\sigma_{-s}(n)&=\frac{x^{1-s}}{2(1-s)}-\frac{x^{1-s}}{s}-\frac{1}{2}\zeta(s)+x\zeta(1+s)+O(1).
\end{align}
Observe that $x^{1-s}=O(1)$ for $s>1$. Use this fact in \eqref{q1} to arrive at \eqref{s=0xi=1caseofgdfeqn}.
\end{proof}

Theorem \ref{logsf} gives the following result too.
\begin{corollary}\label{sigmasfthm}
For $s<0$ and $\alpha>1+s$, we have
\begin{align}\label{sigmasf}
\sum_{n\leq x}\frac{\sigma_{s}(n)}{n^{\alpha}}&=\frac{x^{1+s-{\alpha}}}{s(1-{\alpha})}+\frac{x^{1-{\alpha}}}{1-{\alpha}}\zeta(1-s)+\frac{x^{1+s-{\alpha}}}{1+s-{\alpha}}\zeta({\alpha})+\zeta({\alpha})\zeta({\alpha}-s)+O\left(x^{\lambda}\right),
\end{align}
where $\lambda=\mathrm{max}\left\{{-{\alpha}},\ {1+s-{\alpha}}\right\}$.
\end{corollary}
\begin{proof}
Let $z=1$ in \eqref{logsfeqn} and use the fact from 
 \eqref{en0} 
\begin{align}
E_{1-s}(0)&=-\frac{1}{s},\ \mathrm{for}\ s<0,\
E_{\alpha-s}(0)=-\frac{1}{1+s-\alpha},  \mathrm{for}\ \alpha>1+s,
\end{align}
to arrive at \eqref{sigmasf}.
\end{proof}

We also get \cite[p.~70, Exercise 3]{apostal} as a special case of Theorem \ref{logsf} or Corollary \ref{sigmasfthm}.
\begin{corollary}
Let $\alpha>1$. We have
\begin{align}\label{sigmaeqn}
\sum_{n\leq x}\frac{d(n)}{n^\alpha}&=\frac{x^{1-\alpha}}{(1-\alpha)}+\zeta^2(\alpha)+O\left(x^{1-\alpha}\right).
\end{align}
\end{corollary}
\begin{proof}
Note that, as $s\to0$
\begin{align*}
\frac{x^s}{s}+\zeta(1-s)\to \log(x)+\gamma.
\end{align*}
Let $s\to0$ in \eqref{sigmasf} and use the above expression to see that
\begin{align*}
\sum_{n\leq x}\frac{d(n)}{n^{\alpha}}&=\frac{x^{1-\alpha}}{(1-\alpha)}+\frac{x^{1-\alpha}}{1-\alpha}\zeta(\alpha)+\zeta^2(\alpha)+O\left(x^{1-\alpha}\right).
\end{align*}
Observe that $\frac{x^{1-\alpha}}{1-\alpha}\zeta(\alpha)=O\left(x^{1-\alpha}\right)$ and use this fact in the above equation to arrive at \eqref{sigmaeqn}.
\end{proof}

\section{Concluding Remarks}\label{conclude}
This work arose from our quest to study the series \eqref{parentsum}. We obtained an equivalent representation for this series in Theorem \ref{maintheorem1}. The study of this series allows us to find a generalization of a result of Andrews, Crippa and Simon, i.e., \eqref{secondgen}. In the course of studying \eqref{parentsum}, we encountered a surprising new generalization of the divisor function $\sigma_s(n)$, that is, \eqref{newdivisor}. Several properties of this new divisor function $\sigma_{s,z}(n)$ is obtained in this article. We hope this will instigate further research on the properties of this function.

It will be interesting to obtain results analogous to those obtained by Dixit and Maji for \eqref{DM} for the series
\begin{align}\label{lk}
\sum_{n=1}^{\i}\frac{(b/z;q)_nz^n}{(1-cq^n)^k(bq;q)_n}.
\end{align}
The importance of this proposed study is clearly visible for $k=1$ from the paper of Dixit and Maji \cite{dixitmaji}. For $z\to0$ and $k=1$ of \eqref{lk}, many authors, for example, Uchimura \cite{uchi12}, Dilcher \cite{Dil}, and Yan and Fu \cite{yanfa} studied the finite analogues. Therefore finite analogues of Theorem \ref{maintheorem1} will also be interesting to explore. 

Here we emphasize that Simon-Crippa-Collenberg \cite{scc} showed that the expectation and variance of a certain random variable arising from acyclic digraphs can also be represented in terms of divisor function. Note that in Theorem \ref{maintheorem2}, we obtained a generalization of their identity. Therefore it will be worthwhile to find an application of our Theorem \ref{maintheorem2} in the theory of acyclic digraph similar to that of  Simon, Crippa and Collenberg.

\section{\textbf{Acknowledgements}}
The authors sincerely thanks Professor Atul Dixit for a careful reading of this article, his valuable suggestions and his support throughout this work. They are also thankful to the institution Indian Institute of Technology Gandhinagar for providing them state-of-the-art research facilities. The second author is an institute postdoctoral fellow at IIT Gandhinagar and sincerely thanks the institute and Professor Atul Dixit for financial support.

\end{document}